\documentclass[12pt]{article}
\usepackage{amsmath,amssymb,amsthm}
\usepackage{amsfonts, amsmath}
\usepackage[english]{babel}
\newtheorem{theorem}{Theorem}[section]
\newtheorem{lemma}{Lemma}[section]
\newtheorem{proposition}{Proposition}[section]

\theoremstyle{remark}
\newtheorem{remark}{Remark}[section]
\newtheorem{example}{Example}[section]

\numberwithin{equation}{section}

\newcommand\ep\varepsilon

\newcommand{\R}{\mathbb R}

\DeclareMathOperator{\essinf}{\rm ess\,inf}

\newcommand{\beq}{\begin{equation}}
\newcommand{\eeq}{\end{equation}}

\begin{document}
\title{\textbf{Multiple nodal solutions to a Robin problem with sign-changing potential and locally defined reaction}}
\author{
\bf Umberto Guarnotta, Salvatore A. Marano\thanks{Corresponding Author}\\
\small{Dipartimento di Matematica e Informatica,
Universit\`a degli Studi di Catania,}\\
\small{Viale A. Doria 6, 95125 Catania, Italy}\\
\small{\it E-mail: umberto.guarnotta@gmail.com, marano@dmi.unict.it}\\
\mbox{}\\
\bf Nikolaos S. Papageorgiou\\
\small{Department of Mathematics, National Technical University of Athens,}\\
\small{Zografou Campus, 15780 Athens, Greece}\\
\small{\it E-mail: npapg@math.ntua.gr}
}
\date{}
\maketitle
\centerline{\textit{Dedicated to Professor Carlo Sbordone on the occasion of his seventieth birthday}}
\begin{abstract}
A Robin boundary-value problem with non-homogeneous differential operator, indefinite potential, and reaction defined only near zero is investigated. The existence of one or more nodal solutions is achieved by using truncation, perturbation, and comparison techniques, results from Morse theory, besides variational methods.
\end{abstract}
\vspace{2ex}
\vspace{2ex}

\section{Introduction}
The main purpose of this paper is to investigate both existence and multiplicity of nodal $C^1$-solutions to the following Robin boundary-value problem:

\begin{equation}\label{prob}
\left\{
\begin{array}{ll}
-{\rm div}(a(\nabla u))+\alpha(x)|u|^{p-2}u=f(x, u)\quad & \mbox{in}\quad\Omega,\\
\displaystyle{\frac{\partial u}{\partial n_a}}+\beta(x)|u|^{p-2}u=0\quad & \mbox{on}\quad\partial\Omega.\\
\end{array}
\right.
\end{equation}
Here, $\Omega$ denotes a bounded domain in $\R^N$, $N\geq 3$, with a smooth boundary $\partial\Omega$, the coefficient $\alpha$ is essentially bounded but sign-changing, $\beta$ lies in $C^{0,\gamma}(\partial\Omega)$ and takes non-negative values,  $1<p<+\infty$, the reaction $f:\Omega\times[-\theta,\theta]\to\R$ satisfies Carath\'{e}odory's conditions. Moreover, $a:\R^N\to\R^N$ indicates a strictly monotone map having appropriate regularity and growth properties, while $\frac{\partial}{\partial n_a}$ stands for the co-normal derivative associated with $a$; cf. Section \ref{sectiontwo}.\\
Problem \eqref{prob} exhibits at least three interesting features:
\begin{itemize}
\item[i)] We do not require that $\xi\mapsto a(\xi)$ be $(p-1)$-homogeneous. So, meaningful differential operators, as the $(p,q)$-Laplacian, are incorporated in \eqref{prob}.
\item[ii)] The potential term $u\mapsto\alpha(x)|u|^{p-2}u$ turns out indefinite, because $\alpha$ can change sign.
\item[iii)] $t\mapsto f(x,t)$ is only locally defined, whence its behavior near zero matters, and no conditions at infinity are imposed.
\end{itemize}
Via truncation-perturbation-comparison techniques, results from Morse theory, besides variational methods, we obtain a nodal solution $\hat u\in C^1(\overline{\Omega})$ of \eqref{prob}; see Theorems \ref{thmone}--\ref{thmtwo} below. The case $p>2$ and $a(\xi):=(|\xi|^{p-2}+1)\xi$, namely when the $(p,2)$-Laplacian appears, is examined next in Theorem \ref{ptwocase}, which allows also $f$ to be resonant.

As far as we know, the existence of sign-changing solutions to Robin problems that exhibit difficulties i)--iii)  did not receive much attention up to now. Topics i) and, somehow, iii) have been recently addressed in \cite{PW-AM2016}, while \cite{FMP,MaPa-RLM2018} investigate ii) but for $a(\xi):=|\xi|^{p-2}\xi$. Further items can evidently be found in their bibliographies.

Section \ref{sectionfour} deals with multiplicity. If $f(x,\cdot)$ is odd, Theorem \ref{sequencesol} gives a whole sequence
$\{u_n\}\subseteq C^1(\overline{\Omega})$ of nodal solutions to \eqref{prob} such that $u_n\to 0$ in $C^1(\overline{\Omega})$. The works \cite{HYSZ-NA2015,MaMoPa-RLM2018} contain similar results concerning Dirichlet problems without indefinite potential. All of them exploit an abstract theorem by Kajikiya \cite{Ka-JFA2005}. Once the map $a$ is particularized, we can do without symmetry and still produce two or three nodal $C^1$-solutions; cf. Theorem \ref{spcaseone}, which treats $(p,2)$-Laplace equations, and Theorem \ref{spcasetwo}, where in addition $p=2$. It should be pointed out that one solution always comes from a flow invariance argument patterned after that of \cite{HHLL}, devoted to Neumann's case; see also \cite{PaPa-IJM2014}.
\section{Preliminaries. The map $\xi\mapsto a(\xi)$}\label{sectiontwo}
Let $(X,\Vert\cdot\Vert)$ be a real Banach space. Given a set $V\subseteq X$, write $\overline{V}$ for the closure of $V$, $\partial V$ for the boundary of $V$, and ${\rm int}_X(V)$ or simply ${\rm int}(V)$, when no confusion can arise, for the interior of $V$. If $x\in X$ and $\delta>0$ then
$$B_\delta(x):=\{ z\in X:\;\Vert z-x\Vert<\delta\}\, ,\quad B_\delta:=B_\delta(0)\, .$$
The symbol $(X^*,\Vert\cdot \Vert_{X^*})$ denotes the dual space of $X$, $\langle\cdot,\cdot\rangle$ indicates the duality pairing between $X$ and $X^*$, while $x_n\to x$ (respectively, $x_n\rightharpoonup x$) in $X$ means `the sequence $\{x_n\}$ converges strongly (respectively, weakly) in $X$'. We say that $A:X\to X^*$ is of type $({\rm S})_+$ provided
$$x_n\rightharpoonup x\;\mbox{ in }\; X,\quad
\limsup_{n\to+\infty}\langle A(x_n),x_n-x\rangle\leq 0\quad\implies\quad x_n\to x.$$
The function $\Phi:X\to\mathbb{R}$ is called coercive if $\displaystyle{\lim_{\Vert x\Vert\to+\infty}}\Phi(x)=+\infty$ and weakly sequentially lower semi-continuous when
$$x_n\rightharpoonup x\;\mbox{ in }\; X\quad\implies\quad\Phi(x)\leq\liminf_{n\to\infty}\Phi(x_n).$$
Let $\Phi\in C^1(X)$. The classical Palais-Smale compactness condition for $\Phi$ reads as follows.
\begin{itemize}
\item[$({\rm PS})$] {\it Every sequence $\{x_n\}\subseteq X$ such that $\{\Phi(x_n)\}$ is bounded and  $\displaystyle{\lim_{n\to+\infty}}\Vert \Phi'(x_n) \Vert_{X^*}=0$ has a convergent subsequence.}
\end{itemize}
The next elementary result \cite[Proposition 2.2]{MaPa-PEMS2014} will be employed later.
\begin{proposition}\label{coercivePS}
Suppose $X$ reflexive, $\Phi\in C^1(X)$ coercive, and $\Phi'=A+B$, with $A$ of type $({\rm S})_+$ and $B$ compact. Then $\Phi$ satisfies $({\rm PS})$.
\end{proposition}
Define, for every $c\in\mathbb{R}$,
$$\Phi^c:=\{x\in X:\; \Phi(x)\leq c\}\, ,\quad K_c(\Phi):=K(\Phi)\cap\Phi^{-1}(c)\, ,$$
where, as usual, $K(\Phi)$ denotes the critical set of $\Phi$, i.e., $K(\Phi):=\{x\in X:\,\Phi'(x)=0\}$. 

Given a topological pair $(A,B)$ fulfilling $B\subset A\subseteq X$, the symbol $H_k(A,B)$, $k\in\mathbb{N}_0$, indicates the
${\rm k}^{\rm th}$-relative singular homology group of $(A,B)$ with integer coefficients. If $x_0\in K_c(\Phi)$ is an isolated point of $K(\Phi)$ then
$$C_k(\Phi,x_0):=H_k(\Phi^c\cap V,\Phi^c\cap V\setminus\{x_0\})\, ,\quad k\in
\mathbb{N}_0\, ,$$
are the critical groups of $\Phi$ at $x_0$. Here, $V$ stands for any neighborhood of $x_0$ such that $K(\Phi)\cap\Phi^c\cap V=\{x_0\}$. By excision, this definition does not depend on the choice of $V$. Suppose $\Phi$ satisfies condition $({\rm PS})$, $\Phi|_{K(\Phi)}$ is bounded below, and $c<\displaystyle{\inf_{x\in K(\Phi)}}\Phi(x)$. Put
$$C_k(\Phi,\infty):=H_k(X,\Phi^c)\, ,\quad k\in\mathbb{N}_0\,.$$
The second deformation lemma \cite[Theorem 5.1.33]{GaPa-NA2006} implies that this definition does not depend on the choice of $c$. If $K(\Phi)$ is finite, then setting
$$M(t,x):=\sum_{k=0}^{+\infty}{\rm rank}\, C_k(\Phi,x)t^k\, ,\quad
P(t,\infty):=\sum_{k=0}^{+\infty}{\rm rank}\, C_k(\Phi,\infty)t^k\quad
\forall\, (t,x)\in\mathbb{R}\times K(\Phi)\, ,$$
the following Morse relation holds:
\begin{equation}\label{morse}
\sum_{x\in K(\Phi)}M(t,x)=P(t,\infty)+(1+t)Q(t)\, ,
\end{equation}
where $Q(t)$ denotes a formal series with nonnegative integer coefficients; see for instance \cite[Theorem 6.62]{MoMoPa}. 

Now, let $X$ be a Hilbert space, let $x\in K(\Phi)$, and let $\Phi$ be $C^2$ in a neighborhood of $x$. If $\Phi''(x)$ turns out to be invertible, then $x$ is called non-degenerate. The Morse index $d$ of $x$ is the supremum of the dimensions of the vector subspaces of $X$ on which $\Phi''(x)$ turns out to be negative definite. When $x$ is non-degenerate and with Morse index $d$ one has $C_k(\Phi,x)=\delta_{k,d}\mathbb{Z}$, $k\in\mathbb{N}_0$. The monograph \cite{MoMoPa} represents a general reference on the subject.

Henceforth, $\Omega$ will denote a bounded domain of the real euclidean $N$-space $(\mathbb{R}^N,|\cdot|)$, $N\geq 3$, with a $C^2$-boundary $\partial\Omega$, on which we will employ the $(N-1)$-dimensional Hausdorff measure $\sigma$. The symbol $n(x)$ indicates the outward unit normal vector to $\partial\Omega$ at its point $x$,
$p\in ]1,+\infty[$, $p':=p/(p-1)$, $\Vert\cdot\Vert_s$ with $s\geq 1$ is the usual norm of $L^s(\Omega)$, $X:=W^{1,p}(\Omega)$, and
$$\Vert u\Vert:=\left(\Vert u\Vert_p^p+\Vert\nabla u\Vert_p^p\right)^{1/p},\quad u\in X.$$
Write $p^*$ for the critical exponent of the Sobolev embedding $W^{1,p}(\Omega)\hookrightarrow L^s(\Omega)$. Recall that $p^*=Np/(N-p)$ if $p<N$, $p^*=+\infty$ otherwise, and the embedding turns out to be compact whenever $1\leq s<p^*$. Given $t\in\mathbb{R}$ and $u,v:\Omega\to\mathbb{R}$, $t^\pm:=\max\{\pm t,0\}$, $u^\pm (x):=u(x)^\pm$, $u\leq v$ (resp., $u<v$, etc.) means $u(x)\leq v(x)$ (resp., $u(x)<v(x)$, etc.) for almost every $x\in\Omega$. If $u,v$ belongs to a function space, say $Y$, then we set
$$[u,v]:=\{ w\in Y:u\leq w\leq v\}\, ,\quad Y_+:=\{w\in Y:w\geq 0\}\, .$$
Putting $C_+:= C^1(\overline{\Omega})_+$ and ${\rm int}(C_+):={\rm int}_{C^1(\overline{\Omega})}(C_+)$, one evidently has
\begin{equation*}
{\rm int}(C_+):=\{u\in C_+:u(x)>0\;\;\forall\, x\in\overline{\Omega}\}\, .
\end{equation*}

From now on, $c_1,c_2,\ldots$ denote appropriate positive constants while $\R^+:=\, ]0,+\infty[$.\\
 Let $\omega\in C^1(\R^+)$ and let $\tau\in[1,p[$ satisfy
\begin{equation}\label{assomega}
c_1\leq\frac{t\omega'(t)}{\omega(t)}\leq c_2,\quad c_3t^{p-1}\leq\omega(t)\leq c_4(t^{\tau-1}+t^{p-1}),\quad t\in\R^+.
\end{equation}
The following assumptions on $a:\R^N\to\R^N$ will be posited.
\begin{itemize}
\item[$({\rm a}_1)$] $a(\xi):=a_0(|\xi|)\xi$ for all $\xi\in\R^N$, where $a_0\in C^1(\R^+,\R^+)$, $t\mapsto ta_0(t)$ turns out to be strictly increasing on $\R^+$, and
$$\lim_{t\to 0^+}ta_0(t)=0,\quad\lim_{t\to 0^+}\frac{ta_0'(t)}{a_0(t)}>-1.$$
\item[$({\rm a}_2)$] $\displaystyle{|\nabla a(\xi)|\leq c_5\frac{\omega(|\xi|)}{|\xi|}}$ for every $\xi\in\R^N\setminus\{0\}$.
\item[$({\rm a}_3)$] $\displaystyle{(\nabla a(\xi)y)\cdot y\geq\frac{\omega(|\xi|)}{|\xi|}|y|^2}$ for all $\xi,y\in\R^N$ with $\xi\neq 0$.
\item[$({\rm a}_4)$] If $G_0(t):=\int_0^t sa_0(s)ds$, then there exist $1<\hat q<q\leq p$ such that
$$c_6t^p\leq t^2a_0(t)-\hat q G_0(t)\quad\forall\, t\in\R^+,\quad \lim_{t\to 0^+}\frac{qG_0(t)}{t^q}=c_7,$$
and $t\mapsto G_0(t^{1/q})$ is convex on $[0,+\infty[$.
\end{itemize}
\begin{remark}
Conditions $({\rm a}_1)$--$({\rm a}_3)$ come from Lieberman's nonlinear regularity theory \cite{L} and Pucci-Serrin's maximum principle \cite{PS}.
\end{remark}
\begin{lemma}\label{mapa}
Let $({\rm a}_1)$--$({\rm a}_3)$ be satisfied. Then:
\begin{itemize}
\item[$({\rm i}_1)$] $\xi\mapsto a(\xi)$ is strictly monotone, continuous, and, a fortiori, maximal monotone.
\item[$({\rm i}_2)$] $|a(\xi)|\leq c_8(|\xi|^{\tau-1}+|\xi|^{p-1})$ for every $\xi\in\R^N$.
\item[$({\rm i}_3)$] $\displaystyle{a(\xi)\cdot\xi\geq\frac{c_3}{p-1}|\xi|^p}$ for all $\xi\in\R^N$.
\end{itemize}
\end{lemma}
\begin{proof} Conclusions $({\rm i}_1)$--$({\rm i}_2)$ are obvious. Let us verify $({\rm i}_3)$. Because of $({\rm a}_3)$ and \eqref{assomega} we easily obtain
\begin{equation*}
a(\xi)\cdot\xi=\int_0^1\frac{d}{dt}a(t\xi)\cdot\xi\, dt=\int_0^1(\nabla a(t\xi)\xi)\cdot\xi\, dt\geq\int_0^1\frac{\omega(t|\xi|)}{t|\xi|}|\xi|^2 dt\geq\frac{c_3}{p-1}|\xi|^p
\end{equation*}
for every $\xi\in\R^N$.
\end{proof}
\begin{remark}
Thanks to $({\rm a}_1)$, the function $G_0$ defined in $({\rm a}_4)$ is strictly increasing and convex. Consequently, also the map $G:\R^N\to\R$ given by
\begin{equation*}
G(\xi):=G_0(|\xi|),\quad\xi\in\R^N,
\end{equation*}
turns out convex. Moreover, $G(0)=0$, 
\begin{equation}\label{prop1G}
\nabla G(\xi)=G_0'(|\xi|)\frac{\xi}{|\xi|}=a_0(|\xi|)\xi=a(\xi)\;\;\text{provided}\;\;\xi\neq 0,
\end{equation}
as well as, on account of $({\rm a}_1)$,
\begin{equation}\label{prop2G}
G(\xi)\leq a(\xi)\cdot\xi\quad\forall\,\xi\in\R^N.
\end{equation}
\end{remark}
\begin{lemma}
If $({\rm a}_1)$--$({\rm a}_4)$ hold true, then
\begin{equation}\label{estG}
\frac{c_3}{p(p-1)}|\xi|^p\leq G(\xi)\leq c_9\left[|\xi|^q+|\xi|^p\right],\;\;\xi\in\R^N.
\end{equation}
\end{lemma}
\begin{proof} Through \eqref{prop1G} and $({\rm i}_3)$ in Lemma \ref{mapa} we get
$$G(\xi)=\int_0^1\frac{d}{dt}G(t\xi)\, dt=\int_0^1a(t\xi)\cdot\xi\, dt\geq\int_0^1\frac{c_3}{p-1}|\xi|^pt^{p-1}dt=\frac{c_3}{p(p-1)}|\xi|^p.$$
The other inequality easily follows from $({\rm a}_4)$, \eqref{prop2G}, and $({\rm i}_2)$ of Lemma \ref{mapa}.
\end{proof}
\begin{example}
The functions $a_0$ listed below comply with $({\rm a}_1)$--$({\rm a}_4)$.
\begin{itemize}
\item $a_0(t):=t^{p-2}$ for every $t\in\R^+$. It corresponds to the well-known $p$-Laplacian $\Delta_p$, defined by
$$\Delta_p u:={\rm div}\, (|\nabla u|^{p-2}\nabla u),\;\; u\in X.$$
\item $a_0(t):=t^{p-2}+t^{q-2}$ for all $t\in\R^+$. The associated operator, usually called $(p,q)$-Laplacian, arises in mathematical physics; see, e.g., the survey paper \cite{MaMoDCDS-S}. 
\item $a_0(t):=(1+t^2)^\frac{p-2}{2}$ for every $t\in\R^+$. This function stems form the generalized $p$-mean curvature operator, namely
$$u\mapsto{\rm div}\, \left[(1+|\nabla u|^2)^\frac{p-2}{2}\,\nabla u\right],\;\; u\in X.$$
\item $a_0(t):=t^{p-2}\left(1+\frac{1}{1+t^p}\right)$ for all $t\in\R^+$. It corresponds to the differential operator
$$u\mapsto\Delta_p u+{\rm div}\, \left(\frac{|\nabla u|^{p-2}}{1+|\nabla u|^p}\nabla u\right),\;\; u\in X,$$
which employs in plasticity theory \cite{FG}.
\end{itemize}
\end{example}
Finally, let $A:X\to X^*$ be the nonlinear operator associated with $a$, i.e.,
$$\langle A(u),v\rangle:=\int_\Omega a(\nabla u)\cdot\nabla v\, dx\quad\forall\, u,v\in X.$$
Proposition 3.5 in \cite{GaPa2008} ensures that $A$ is bounded, continuous, monotone, as well as of type $({\rm S})_+$. Moreover, via the nonlinear Green's identity \cite[Theorem 2.4.54]{GaPa-NA2006} we easily have
\begin{remark}\label{equivprob}
If $u\in X$, $w\in L^{p'}(\Omega)$, and $\beta\in C^{0,\gamma}(\partial\Omega,\R^+_0)$ then
$$\langle A(u),v\rangle+\int_{\partial\Omega}\beta(x)|u(x)|^{p-2}u(x)v(x)d\sigma=\int_\Omega w(x)v(x)dx,\quad v\in X,$$
is equivalent to  
$$-{\rm div}\, (a(\nabla u))=w(x)\quad\mbox{in}\quad\Omega,\quad\frac{\partial u}{\partial n_a}+\beta(x)|u|^{p-2}u=0\quad\mbox{on}\quad
\partial\Omega.$$
Here, $\frac{\partial u}{\partial n_a}$ denotes the co-normal derivative of $u$, defined extending the map $v\mapsto a(\nabla v)\cdot n$ from $C^1(\overline{\Omega})$ to $X$.
\end{remark}
Let $\beta\in C^{0,\gamma}(\partial\Omega,\R^+_0)$, let $h:\Omega\times\R\to \R$ be a  Carath\'{e}odory function such that
$$|h(x,t)|\leq C(1+|t|^{s-1})\quad\forall\, (x,t)\in\Omega\times\R,$$
where $C>0$, $1\leq s\leq p^*$, and let $H(x,t):=\int_0^t h(x,\tau)\, d\tau$. Consider the $C^1$-functional $\varphi_h: X\to\R$ defined by
$$\varphi_h(u):=\int_\Omega G(\nabla u(x))\, dx+\frac{1}{p}\int_{\partial\Omega}\beta(x)|u(x)|^pd\sigma
-\int_\Omega H(x, u(x))\, dx,\quad u\in X.$$
A relation between local minimizers of $\varphi_h$ in  $C^1(\overline{\Omega})$ and in $X$ occurs \cite[Proposition 8]{PaRaANS2016}.
\begin{proposition}\label{locmin}
Suppose $u_0\in X$ is a local $C^1(\overline{\Omega})$-minimizer to $\varphi_h$. Then  $u_0\in C^{1,\eta}(\overline{\Omega})$ and is a local minimizer of $\varphi_h$.
\end{proposition}
We shall employ some facts about the spectrum of the operator
$$u\mapsto-\Delta_q u+\alpha(x)|u|^{q-2}u$$
in $X$ with homogeneous Robin boundary conditions. So, consider the eigenvalue problem
\begin{equation}\label{eigen}
-\Delta_q u+\alpha(x)|u|^{q-2}u=\lambda |u|^{q-2}u\quad\mbox{in}\quad\Omega,\quad\frac{\partial u}{\partial n_q}+\beta(x) |u|^{q-2}u=0\quad\mbox{on}\quad \partial\Omega\, ,
\end{equation}
where, henceforth,
\begin{equation}\label{abeta}
\alpha\in L^\infty(\Omega)\;\;\mbox{and}\;\;\beta\in C^{0,\gamma}(\partial\Omega,\R^+_0)\;\;\mbox{with}\;\;
\gamma\in\ ]0,1[\, .
\end{equation}
Define
\begin{equation}\label{defE}
{\cal E}_q (u):=\Vert\nabla u\Vert_q^q+\int_\Omega\alpha(x)|u(x)|^qdx+\int_{\partial\Omega}\beta(x)|u(x)|^q d\sigma\quad\forall\, u\in X.
\end{equation}
The Liusternik-Schnirelman theory provides a strictly increasing sequence $\{\hat\lambda_n(q,\alpha,\beta)\}$ of eigenvalues for \eqref{eigen}. As in \cite{MuPa-ASNSP2012,PaRa-JDE2014}, one has
\begin{itemize}
\item[$({\rm p}_1)$] \textit{$\hat\lambda_1(q,\alpha,\beta)$ turns out to be isolated and simple. Further, $\hat\lambda_1(q,\alpha,\beta)=\displaystyle{\inf_{u\in X\setminus\{0\}}\frac{{\cal E}_q(u)}{\Vert u\Vert_q^q}}$.}
\item[$({\rm p}_2)$] \textit{There is an eigenfunction $\hat u_1(q,\alpha,\beta)\in{\rm int}(C_+)$ associated with $\hat\lambda_1 (q,\alpha,\beta)$ such that $\Vert  u_1(q,\alpha,\beta)\Vert_q=1$.}
\item[$({\rm p}_3)$] \textit{Write $U:=\{ u\in X:\; \Vert u\Vert_q=1\}$ and
$$\hat\Gamma:=\{\hat\gamma\in C^0([-1,1],U):\hat\gamma(-1)=-\hat u_1(q,\alpha,\beta),\;
\hat\gamma(1)=\hat u_1(q,\alpha,\beta)\}.$$
Then $\displaystyle{\hat\lambda_2(q,\alpha,\beta)=\inf_{\hat\gamma\in\hat\Gamma}\max_{t\in [-1,1]}{\cal E}_q(\hat\gamma(t))}$.}
\end{itemize}
Evidently, the set $U_C:=\{u\in C^1(\overline{\Omega}): \|u\|_q=1\}$ turns out dense in $U$. Moreover, if
$$\hat\Gamma_C:=\{\hat\gamma\in C^0([-1,1],U_C):\;\hat\gamma(-1)=-\hat u_1(q,\alpha,\beta),\;\hat\gamma(1)=\hat u_1(q,\alpha,\beta)\}.$$
then (cf. \cite[Lemma 2.1]{MaPa-MM2016})
\begin{lemma}\label{density}
$\hat\Gamma_C$ is dense in $\hat\Gamma$ with respect to the usual norm of $C^0([-1,1],X)$.
\end{lemma}
Let $q:=2$. Denote by $E(\hat\lambda_n)$ the eigenspace coming from $\hat\lambda_n:=\hat\lambda_n(2,\alpha,\beta)$. It is known \cite{DaMaPa-JMAA2016,MaPa-JMAA2016} that $H^1(\Omega)=\overline{\oplus_{n=1}^\infty E(\hat\lambda_n)}$ and that
\begin{itemize}
\item[$({\rm p}_4)$] \textit{For every $n\geq 2$ one has
\begin{equation*}
\hat\lambda_n=\sup\left\{\frac{{\cal E}_2(u)}{\Vert u\Vert_2^2}:u\in\bar H_n,\, u\neq 0\right\}
=\inf\left\{\frac{{\cal E}_2(u)}{\Vert u\Vert_2^2}:u\in\hat H_n,\, u\neq 0\right\}\, ,
\end{equation*}
where
$$\bar H_n:=\oplus_{i=1}^n E(\hat\lambda_i),\quad\hat H_n:=\oplus_{i=n}^\infty E(\hat\lambda_i)\, .$$
}
\end{itemize}
\section{Nodal solutions: existence}
To avoid unnecessary technicalities, `for every $x\in\Omega$' will take the place of `for almost every $x\in\Omega$' while $C_1,C_2, \ldots$ indicate positive constants arising from the context.

Let $\underline{u},\overline{u}\in C^0(\overline{\Omega})\cap W^{1,p}(\Omega)$ satisfy
$$\displaystyle{\max_{x\in\overline{\Omega}}}\,\underline{u}(x)<0<\displaystyle{\min_{x\in\overline{\Omega}}}\,\overline{u}(x),$$ 
let $\theta:=\max\{\Vert\underline{u}\Vert_\infty,\Vert\overline{u}\Vert_\infty\}$, and let $f:\Omega\times[-\theta,\theta]\to\R$ be a Carath\'eodory function. We shall make the following two hypotheses throughout the paper.
\begin{itemize}
\item[(${\rm f}_1$)] One has
$$\left\{
\begin{array}{ll}
\langle A(\underline{u}),v\rangle+\int_\Omega\alpha|\underline{u}|^{p-2}\underline{u}v\, dx\leq 0\leq\int_\Omega f(x,\underline{u}) v\, dx,\\
\phantom{}\\
\int_\Omega f(x,\overline{u}) v\, dx\leq 0\leq\langle A(\overline{u}),v\rangle+\int_\Omega\alpha|\overline{u}|^{p-2}\overline{u}v\, dx
\end{array}
\right.\quad\forall\, v\in X_+\, .$$
Alternatively,
$$\left\{
\begin{array}{ll}
\langle A(\underline{u}),v\rangle\leq 0\leq\int_\Omega\left[ f(x,\underline{u})-\alpha|\underline{u}|^{p-2}\underline{u}\right]
v\, dx\\
\phantom{}\\
\int_\Omega\left[ f(x,\overline{u})-\alpha|\overline{u}|^{p-2}\overline{u}\right] v\, dx\leq 0\leq\langle A(\overline{u}),v\rangle
\end{array}
\right.\quad\forall\, v\in X_+\, .$$
\item[(${\rm f}_2$)] There exists $a_\theta\in L^\infty(\Omega)$ such that $|f(x, t)|\leq a_\theta(x)$ in $\Omega\times[-\theta,\theta]$.
\end{itemize}
\begin{remark}
Condition (${\rm f}_1$) evidently entails either $f(\cdot,\overline{u})\leq 0\leq f(\cdot,\underline{u})$ or
$$f(\cdot,\overline{u})-\alpha|\overline{u}|^{p-2}\overline{u}\leq 0\leq f(\cdot,\underline{u})-\alpha|\underline{u}|^{p-2}\underline{u}.$$
\end{remark}
Different behaviors of $f$ at zero will instead be investigated to get nodal solutions of \eqref{prob}.
\subsection{The $(q-1)$-sub-linear case}
For $q,\hat q$ as in $({\rm a}_4)$ and uniformly with respect to $x\in\Omega$, assume that:
\begin{itemize}
\item[(${\rm f}_3$)] $\displaystyle{\lim_{t\to 0}\frac{f(x,t)}{|t|^{q-2}t}}=+\infty$.
\item[(${\rm f}_4$)] $\displaystyle{\lim_{t\to 0}\frac{\hat q F(x,t)-f(x,t)t}{|t|^p}}>\left(1-\frac{\hat q}{p}\right)
\Vert\alpha\Vert_\infty$.
\end{itemize}
\begin{example}
Let $\alpha\neq 0$. If $\underline{u}(x):=-\Vert\alpha\Vert_\infty$, $\overline{u}(x):=\Vert\alpha\Vert_\infty$, and
$$f(x,t):=C_1|t|^{q-2-\ep}t-C_2|t|^{q-2}t,\;\; (x,t)\in\Omega\times[-\theta,\theta],$$
where $\ep>0$ is small while $C_1,C_2>0$ comply with
$f(\cdot,\overline{u})\leq-\Vert\alpha\Vert_\infty^p\leq f(\cdot,\underline{u}),$
then (${\rm f}_1$)--(${\rm f}_4$) hold true.
\end{example}
\begin{remark}
Because of (${\rm f}_2$)--(${\rm f}_3$), given $r>p$, to each $\eta>0$ there corresponds $C_\eta>0$ fulfilling
\begin{equation}\label{defeta}
f(x,t)t\geq\eta|t|^q-C_\eta|t|^r\quad\forall\, (x,t)\in\Omega\times [-\theta,\theta].
\end{equation}
\end{remark}
Define, provided $(x,t)\in\Omega\times\R$, 
\begin{equation}\label{defh}
h(x,t):=\left\{
\begin{array}{ll}
\left(\eta|\underline{u}(x)|^{q-2}-C_\eta|\underline{u}(x)|^{r-2}+|\underline{u}(x)|^{p-2}\right)\underline{u}(x) & \text{when}\; t<\underline{u}(x),\\
\eta|t|^{q-2}t-C_\eta|t|^{r-2}t+|t|^{p-2}t & \text{if}\;\;\underline{u}(x)\leq t\leq\overline{u}(x),\\
\eta\overline{u}(x)^{q-1}-C_\eta\overline{u}(x)^{r-1}+\overline{u}(x)^{p-1} & \text{when}\;\;t>\overline{u}(x),\\
\end{array}
\right.
\end{equation}
and $H(x,t):=\int_0^t h(x,s)\, ds$, besides
\begin{equation}\label{defgamma}
b(x,t):=\left\{
\begin{array}{ll}
\beta(x)|\underline{u}(x)|^{p-2}\underline{u}(x) & \text{for}\;\;t<\underline{u}(x),\\
\beta(x)|t|^{p-2}t & \text{if}\;\underline{u}(x)\leq t\leq\overline{u}(x),\\
\beta(x)\overline{u}(x)^{p-1} & \text{for}\;\; t>\overline{u}(x),\\
\end{array}
\right.
\end{equation}
and $B(x,t):=\int_0^t b(x,s)\, ds$, where $(x,t)\in\partial\Omega\times\R$. Consider the auxiliary Robin problem
\begin{equation}\label{auxprob}
\left\{
\begin{array}{ll}
-{\rm div}\, (a(\nabla u))+(|\alpha(x)|+1)|u|^{p-2}u=h(x,u)\;\;\text{in}\;\;\Omega,\\
\phantom{}\\
\displaystyle{\frac{\partial u}{\partial n_a}}+b(x,u)=0\;\;\text{on}\;\;\partial\Omega.\\
\end{array}
\right.
\end{equation}
\begin{lemma}\label{lemaux}
Let $({\rm a}_1)$--$({\rm a}_4)$, \eqref{abeta}, $({\rm f}_1)$--$({\rm f}_3)$  be satisfied. Then \eqref{auxprob} possesses a unique positive solution $u_ *\in[0,\overline{u}]\cap{\rm int}(C_+)$ and a unique negative solution $v_*\in [\underline{u},0]\cap(-{\rm int}(C_+))$.
\end{lemma}
\begin{proof}
Thanks to \eqref{estG}, \eqref{defh}, and \eqref{defgamma}, the $C^1$-functional $\psi_h:X\to\R$ defined by
$$\psi_h(u):=\int_\Omega G(\nabla u)\, dx+\frac{1}{p}\int_\Omega(|\alpha|+1)|u|^pdx+\int_{\partial\Omega} B(x,u)\, d\sigma-\int_\Omega H(x,u^+)\, dx,\;\; u\in X,$$
is coercive. A standard argument, which exploits Sobolev's embedding theorem and the compactness of the trace operator, ensures that $\psi_h$ is weakly sequentially lower semicontinuous. Hence,
\begin{equation}\label{ustar}
\inf_{u\in X}\psi_h(u)=\psi_h(u_*)
\end{equation}
for some $u_*\in X$. One has $u_*\neq 0$. Indeed, $({\rm a}_4)$ yields $\delta>0$ fulfilling
\begin{equation*}
G_0(t)<(c_7+1)\frac{t^q}{q}\quad\forall\, t\in\ ]0,\delta].
\end{equation*}
Obviously, we may suppose $\delta\leq\min\{1,\min_{x\in\overline{\Omega}}\overline{u}(x)\}$. If
$\rho>0$ is so small that
\begin{equation*}
0<\rho\hat u_1(x)\leq\delta,\quad \rho|\nabla\hat u_1(x)|\leq\delta,\quad x\in\overline{\Omega},
\end{equation*} 
where $\hat u_1$ comes from $({\rm p}_2)$ with $\alpha_0:=|\alpha|/c_7$ and $\beta_0:=\beta/c_7$ in place of $\alpha$ and
$\beta$ (vide Section \ref{sectiontwo}), respectively, then
\begin{equation*}
\begin{split}
\psi_h(\rho\hat u_1) & \leq(c_7+1)\frac{\rho^q}{q}\left[\Vert\nabla\hat u_1\Vert_q^q+\int_\Omega\alpha_0\hat u_1^q\, dx+
\int_{\partial\Omega}\beta_0\hat u_1^q\, d\sigma\right]-\frac{\eta}{q}\rho^q+\frac{C_\eta}{r}\Vert\hat u_1\Vert_r^r\rho^r\\
& =\frac{\rho^q}{q}\left[(c_7+1)\hat\lambda_1(q,\alpha_0,\beta_0)-\eta\right]+\frac{C_\eta}{r}\Vert\hat u_1\Vert_r^r\rho^r,
\end{split}
\end{equation*}
because $q\leq p$, $\delta\leq 1$, $\Vert\hat u_1\Vert_q=1$. Choosing $\eta>(c_7+1)\hat\lambda_1(q,\alpha_0,\beta_0)$ in \eqref{defeta}, recalling that $r>q$, and decreasing $\rho$ when necessary, entails $\psi_h(\rho\hat u_1)<0$, whence $u_*\neq 0$, as an easy contradiction argument shows.

Through \eqref{ustar} we next get $\psi_h'(u_*)=0$, namely
\begin{equation}\label{uscrit}
\langle A(u_*),w\rangle+\int_\Omega(|\alpha|+1)|u_*|^{p-2}u_*w dx+\int_{\partial\Omega} b(x,u_*)w d\sigma
=\int_\Omega h(x,u_*^+)w dx\;\;\forall\, w\in X.
\end{equation}
Put $w:=u_*^-$ in \eqref{uscrit} and exploit  $({\rm i}_3)$ of Lemma \ref{mapa} to arrive at
\begin{equation*}
\frac{c_3}{p-1}\Vert\nabla u_*^-\Vert_p^p+\Vert u_*^-\Vert_p^p\leq 0.
\end{equation*}
Therefore, $u_*\geq 0$. Now, if $w:=(u_*-\overline{u})^+$ then \eqref{uscrit}, together with \eqref{defeta}, $({\rm f}_1)$, and \eqref{abeta}, produce
\begin{equation*}
\begin{split}
\langle A(u_*), & (u_*-\overline{u})^+\rangle+ \int_\Omega(|\alpha|+1)u_*^{p-1}(u_*-\overline{u})^+ dx
+\int_{\partial\Omega}\beta\,\overline{u}^{p-1}(u_*-\overline{u})^+ d\sigma\\
& =\int_\Omega(\eta\,\overline{u}^{q-1}-C_\eta\overline{u}^{r-1}+\overline{u}^{p-1})(u_*-\overline{u})^+dx\\
& \leq\int_\Omega[f(x,\overline{u})+\overline{u}^{p-1}](u_*-\overline{u})^+dx\\
& \leq\langle A(\overline{u}),(u_*-\overline{u})^+\rangle+ \int_\Omega(|\alpha|+1)\overline{u}^{p-1}(u_*-\overline{u})^+ dx+
\int_{\partial\Omega}\beta\,\overline{u}^{p-1}(u_*-\overline{u})^+ d\sigma.
\end{split}
\end{equation*}
This forces
$$\langle A(u_*)- A(\overline{u}),(u_*-\overline{u})^+\rangle+ \int_\Omega(|\alpha|+1)(u_*^{p-1}-\overline{u}^{p-1})(u_*-\overline{u})^+ dx\leq 0,$$
i.e., $u_*\leq\overline{u}$. Summing up, both $u_*\in[0,\overline{u}]\setminus\{0\}$ and, by \eqref{uscrit} again,
\begin{equation}\label{auxone}
\langle A(u_*),w\rangle+\int_\Omega|\alpha|\, u_*^{p-1}w dx+\int_{\partial\Omega}\beta\, u_*^{p-1}w d\sigma
=\int_\Omega(\eta u_*^{q-1}-C_\eta u_*^{r-1})w dx\;\;\forall\, w\in X.
\end{equation}
Proposition 7 in \cite{PaRaANS2016} ensures that $u_*\in L^\infty(\Omega)$, whence $u_*\in C_+\setminus\{0\}$ thanks to Lieberman's regularity results \cite{L}. Let $\alpha_\theta\in\R^+$ satisfy
\begin{equation}\label{auxtwo}
\eta t^{q-1}-C_\eta t^{r-1}\geq-\alpha_\theta t^{p-1},\quad t\in [0,\theta].
\end{equation}
Because of Remark \ref{equivprob},  from \eqref{auxone}--\eqref{auxtwo} it follows
$${\rm div}\, a(\nabla u_*(x))\leq\left(\Vert\alpha\Vert_\infty+\alpha_\theta\right) u_*(x)^{p-1}\;\;\text{a.e. in}\;\;\Omega.$$
Thus, Pucci-Serrin's maximum principle \cite[p. 120]{PS} yields $u_*\in{\rm int}(C_+)$.

Let us now come to uniqueness. Suppose $\hat u\in[0,\overline{u}]\cap{\rm int}(C_+)$ is another solution to \eqref{auxprob}. Define, provided $u\in L^1(\Omega)$,
\begin{equation*}
J(u):=\left\{
\begin{array}{ll}
\int_\Omega G(\nabla u^{\frac{1}{q}}) dx+\frac{1}{p}\left(\int_\Omega|\alpha|\, u^{\frac{p}{q}}dx +\int_{\partial\Omega}\beta\, u^{\frac{p}{q}}d\sigma\right) & \text{if}\; u\geq 0,\; u^{\frac{1}{q}}\in X,\\
+\infty & \text{otherwise.}
\end{array}
\right.
\end{equation*}
The reasoning made in \cite[pp. 1219--1220]{PW-AM2016} shows here that $u\mapsto\int_\Omega G(\nabla u^{\frac{1}{q}}) dx$ turns out convex. Since $p\geq q$ and $\beta\geq 0$, the same holds for $J$. Via Fatou's lemma we see that $J$ is lower semicontinuous. A simple computation chiefly based on \cite[Theorem 2.4.54]{GaPa-NA2006} gives
$$J'(u_*^q)(w)=\frac{1}{q}\int_\Omega\frac{-{\rm div}\, a(\nabla u_*)+|\alpha|\, u_*^{p-1}}{u_*^{q-1}}\, w\, dx,$$
$$J'(\hat u^q)(w)=\frac{1}{q}\int_\Omega\frac{-{\rm div}\, a(\nabla\hat u)+|\alpha|\, \hat u^{p-1}}{\hat u^{q-1}}\, w\, dx$$
for all $w\in C^1(\overline{\Omega})$ (which is dense in $X$), while the monotonicity of $J'$ entails
\begin{equation*}
\begin{split}
0 & \leq\int_\Omega\left(\frac{-{\rm div}\, a(\nabla u_*)+|\alpha|\, u_*^{p-1}}{u_*^{q-1}}-\frac{-{\rm div}\, a(\nabla\hat u)+
|\alpha|\,\hat u^{p-1}}{\hat u^{q-1}}\right)(u_*^q-\hat u^q)\, dx\\
& =\int_\Omega\left(\frac{\eta u_*^{q-1}-C_\eta u_*^{r-1}}{u_*^{q-1}}-\frac{\eta \hat u^{q-1}-C_\eta\hat u^{r-1}}{\hat u^{q-1}}\right)(u_*^q-\hat u^q)\, dx\\
& =-C_\eta\int_\Omega(u_*^{r-q}-\hat u^{r-q})(u_*^q-\hat u^q)\, dx\leq 0
\end{split}
\end{equation*}
as $q<r$. Consequently, $u_*=\hat u$. Working similarly produces a solution $v_*$ to \eqref{auxprob} with the asserted properties.
\end{proof}
Now, consider the sets
$$\Sigma_+:=\{u\in X\setminus\{0\}:0\leq u\leq\overline{u},\; u\mbox{ solves }\eqref{prob}\},$$
$$\Sigma_-:=\{u\in X\setminus\{0\}:\underline{u}\leq u\leq 0,\; u\mbox{ solves }\eqref{prob}\}.$$
Standard arguments show that:
\begin{itemize}
\item $\Sigma_+\subseteq{\rm int}(C_+)$ while $\Sigma_-\subseteq-{\rm int}(C_+)$ (cf. for instance the above proof);
\item $\Sigma_+$ is downward directed and $\Sigma_-$ is upward directed (see, e.g., \cite[Proposition 7]{PaRaRe-DCDS2017}).
\end{itemize}
\begin{lemma}\label{ustarvstar}
Under the hypotheses of Lemma \ref{lemaux}, one has 
$$u_*\leq u\quad\forall\, u\in\Sigma_+\, ,\quad u\leq v_*\quad\forall\, u\in\Sigma_-\, .$$
\end{lemma}
\begin{proof}
Pick any $u\in\Sigma_+$. Bearing in mind \eqref{defh}--\eqref{defgamma}, define
$$h_+(x,t):=\left\{
\begin{array}{ll}
h(x,t^+) & \mbox{if } t\leq u(x)\\
h(x,u(x)) & \mbox{otherwise}
\end{array}
\right.,\;\; H_+(x,t):=\int_0^t h_+(x,s)ds,\;\; (x,t)\in\Omega\times\R,$$
$$b_+(x,t):=\left\{
\begin{array}{ll}
b(x,t^+) & \mbox{if } t\leq u(x)\\
b(x,u(x)) & \mbox{otherwise}
\end{array}
\right.,\;\; B_+(x,t):=\int_0^t b_+(x,s)ds,\;\; (x,t)\in\partial\Omega\times\R.$$
The associated functional
$$\psi_+(w):=\int_\Omega G(\nabla w)dx+\frac{1}{p}\int_\Omega(|\alpha|+1)|w|^p dx+\int_{\partial\Omega} B_+(x,w) d\sigma-\int_\Omega H_+(x,w) dx$$
is evidently $C^1$, weakly sequentially lower semi-continuous, and coercive. So, there exists $u_0\in X$ such that
$$\psi_+(u_0)=\inf_{w\in X}\psi_+(w).$$
From $q\leq p<r$ it follows, as before, $\psi_+(u_0)<0=\psi_+(0)$, namely $u_0\neq 0$. Moreover, $u_0\in[0,u]$, which entails
\begin{equation*}
\langle A(u_0),w\rangle+\int_\Omega(|\alpha|+1)u_0^{p-1}w\, dx+\int_{\partial\Omega} b(x,u_0)w\, d\sigma=
\int_{\Omega} h(x,u_0)w\, dx,\quad w\in X.
\end{equation*}
Through Lemma \ref{lemaux} we thus have $u_0=u_*$ and, a fortiori, $u_*\leq u$. The remaining proof is analogous.  
\end{proof}
\begin{proposition}\label{propone}
If $({\rm a}_1)$--$({\rm a}_4)$, \eqref{abeta}, and $({\rm f}_1)$--$({\rm f}_3)$  hold then, there exists $u_+\in\Sigma_+$ (resp., $v_-\in\Sigma_-$) such that $u_+\leq u$ for all $u\in\Sigma_+$ (resp., $u\leq v_-$ for all $u\in\Sigma_-$).
\end{proposition}
\begin{proof}
Both arguments are similar. Hence, we shall only present those involving $u_+$. Since $\Sigma_+$ is downward directed, \cite[Lemma 3.10]{HuPa} gives a sequence $\{u_n\}\subseteq\Sigma_+$, which fulfills 
$$u_{n+1}\leq u_n\;\;\forall\, n\in\mathbb{N},\quad\inf_{n\in\mathbb{N}}u_n=\inf\Sigma_+.$$
Consequently, $0\leq u_n\leq\Vert u_1\Vert_\infty$ besides
\begin{equation}\label{un}
\langle A(u_n),w\rangle+\int_\Omega|\alpha|u_n^{p-1}w\, dx+\int_{\partial\Omega}\beta u_n^{p-1}w\, d\sigma=
\int_{\Omega} f(x,u_n)w\, dx,\quad w\in X,
\end{equation}
for every $n\in\mathbb{N}$. Now, put $w:=u_n$ in \eqref{un} and exploit $({\rm i}_3)$ of Lemma \ref{mapa} to verify that $\{u_n\}\subseteq X$ turns out bounded. Let $u_+\in X_+$ satisfy
\begin{equation}\label{lims}
u_n\rightharpoonup u_+\;\;\text{in}\;\; X,\;\; u_n\to u_+\;\;\text{in both}\;\; L^p(\Omega)\;\;\text{and}\;\; L^p(\partial\Omega),
\end{equation}
where a subsequence is considered if necessary. Combining \eqref{un} written for $w:=u_n-u_+$ with \eqref{lims} entails
$$\lim_{n\to+\infty}\langle A(u_n),u_n-u_+\rangle=0,$$
whence $u_n\to u_+$ in $X$, because $A$ enjoys the $({\rm S})_+$-property. Due to \eqref{un} again, this ensures that $u_+$ solves \eqref{prob}. On the other hand, by Lemma \ref{ustarvstar}, from $\{u_n\}\subseteq\Sigma_+$ it follows $u_*\leq u_n$ for all $n\in\mathbb{N}$. Hence, $u_*\leq u_+$ and, a fortiori, $u_+\in\Sigma_+$. Noting that $u_+=\displaystyle{\inf_{n\in\mathbb{N}}}u_n$ completes the proof.
\end{proof}
\begin{remark}\label{nodsol}
On account of Proposition \ref{propone}, every solution $u\in [v_-,u_+]\setminus\{v_-,0,u_+\}$ of \eqref{prob} must be nodal.  
\end{remark}
For $u_+,v_-$ as above and $\hat\alpha>\Vert\alpha\Vert_\infty$, define
\begin{equation}\label{defhatf}
\hat f(x,t):=\left\{
\begin{array}{ll}
f(x,v_-(x))+\hat\alpha|v_-(x)|^{p-2}v_-(x) & \text{when}\; t<v_-(x),\\
f(x,t)+\hat\alpha |t|^{p-2}t & \text{if}\;\; v_-(x)\leq t\leq u_+(x),\\
f(x,u_+(x))+\hat\alpha u_+(x)^{p-1} & \text{when}\;\;t>u_+(x),\\
\end{array}
\right. 
\end{equation}
$$\hat F(x,t):=\int_0^t \hat f(x,s)\, ds,$$
provided $(x,t)\in\Omega\times\R$, besides
\begin{equation}\label{defhatbeta}
\hat b(x,t):=\left\{
\begin{array}{ll}
\beta(x)|v_-(x)|^{p-2} v_-(x) & \text{for}\;\;t<v_-(x),\\
\beta(x)|t|^{p-2}t & \text{if}\;v_-(x)\leq t\leq u_+(x),\\
\beta(x) u_+(x)^{p-1} & \text{for}\;\; t>u_+(x),\\
\end{array}
\right.
\end{equation}
$$\hat B(x,t):=\int_0^t\hat b(x,s)\, ds,$$
where $(x,t)\in\partial\Omega\times\R$. A standard computation, which exploits $({\rm i}_3)$ in Lemma \ref{mapa}, the choice of $\hat\alpha$, and \eqref{defhatf}--\eqref{defhatbeta}, guarantees that the $C^1$-functionals
\begin{equation*}
\hat\varphi (u):=\int_\Omega G(\nabla u) dx+\frac{1}{p}\int_\Omega(\alpha+\hat\alpha)|u|^p dx+\int_{\partial\Omega}\hat B(x,u) d\sigma-\int_\Omega\hat F(x,u) dx,\; u\in X,
\end{equation*}
\begin{equation*}
\hat\varphi_\pm (u):=\int_\Omega G(\nabla u) dx+\frac{1}{p}\int_\Omega(\alpha+\hat\alpha)|u|^p dx+\int_{\partial\Omega} \hat B(x,u) d\sigma-\int_\Omega\hat F(x,u^\pm) dx,\; u\in X,
\end{equation*}
are coercive; so, by Proposition \ref{coercivePS}, they comply with condition (PS). Moreover,
\begin{lemma}\label{prophatphi}
Let $({\rm a}_1)$--$({\rm a}_4)$, \eqref{abeta}, and $({\rm f}_1)$--$({\rm f}_3)$  be satisfied. Then:
\begin{itemize}
\item[$({\rm j}_1)$] $K(\hat\varphi)\subseteq [v_-,u_+]\cap C^1(\overline{\Omega}).$
\item[$({\rm j}_2)$] $u_+$ and $v_-$ are local minimizers of $\hat\varphi$.
\item[$({\rm j}_3)$] $K(\hat\varphi_+)=\{0,u_+\}$ and $K(\hat\varphi_-)=\{0,v_-\}$.
\end{itemize}
\end{lemma}
\begin{proof}
Reasoning as before (cf. Lemma \ref{lemaux}) we can easily check $({\rm j}_1)$. A known argument (see., e.g., \cite[Lemma 3.2]{MaPa-MM2016} or \cite[p. 1227, Claim 2]{PW-AM2016}), chiefly based on Proposition \ref{locmin}, yields $({\rm j}_2)$.  Finally, concerning $({\rm j}_3)$, let us simply note that the obvious inclusion $K(\hat\varphi_+)\subseteq [0,u_+]$ forces $K(\hat\varphi_+)=\{0,u_+\}$ by extremality of $u_+$; cf. Proposition \ref{propone}. The same goes for $K(\hat\varphi_-)=\{0,v_-\}$.
\end{proof}
\begin{lemma}\label{prophatphix}
Under $({\rm f}_2)$--$({\rm f}_4)$ one has $C_k(\hat\varphi,0)=0$ for all $k\in\mathbb{N}_0$.
\end{lemma}
\begin{proof}
It is rather delicate, but essentially analogous to the one made in \cite[Proposition 4.1]{PW-AM2016}. We shall present here a simpler trick. Fix $r>p$ and $\eta>0$. Assumptions $({\rm f}_2)$--$({\rm f}_3)$ give $C_1>0$ such that
\begin{equation*}
\hat F(x,t)\geq\eta|t|^q-C_1|t|^r,\;\;(x,t)\in\Omega\times\R.
\end{equation*}
Because of \eqref{estG} this implies
\begin{equation*}
\hat\varphi(tu)\leq c_9\left(t^q\Vert\nabla u\Vert_q^q+t^p\Vert\nabla u\Vert_p^p\right)+C_2t^p\Vert u\Vert_p^p+C_1t^r\Vert u\Vert_r^r-\eta t^q\Vert u\Vert_q^q
\end{equation*}
for every $t>0$, $u\in X$. Since $\eta$ was arbitrary while $q\leq p<r$, if $u\neq 0$ then there exists $t^*\in\  ]0,1[$ (which may depend on $u$) fulfilling $\hat\varphi(tu)<0$ whatever $t\in\ ]0,t^*[$. Define $t_1:= \sup\{ t \in [0,1]: \hat\varphi(tu) < 0\}$ as well as
\begin{equation*}
t_2:=\left\{
\begin{array}{ll}
\inf\{ t \in [0,1]: \hat\varphi(tu) \geq 0\} & \text{ when } \{ t \in [0,1]: \hat\varphi(tu) \geq 0\}\neq\emptyset,\\
1 & \text{ otherwise.}
\end{array}
\right.
\end{equation*}
We will show that $t_1\leq t_2$. By contradiction, suppose $t_2 < t_1$. Let $t_0\in\ ]0,1] $ satisfy $\hat\varphi(v)= 0$, where $v:= t_0 u$. Simple calculations based on (${\rm f}_2$)--(${\rm f}_4$) and \eqref{defeta} yield
\begin{equation*}
\hat q F(x,t)-f(x,t)t\geq\left[\left(1-\frac{\hat q}{p}\right)\Vert\alpha\Vert_\infty+k\right]|t|^p-C_3|t|^r,\;\; (x,t)\in\Omega\times[-\theta,\theta],
\end{equation*}
for some $k>0$, $C_3>0$. Thanks to $({\rm a}_4) $, this inequality entails
\begin{equation*}
\begin{split}
t_0\frac{d}{dt}\hat\varphi(tu)\lfloor_{t = t_0} & =\langle\hat\varphi'(v),v\rangle
=\langle\hat\varphi'(v),v\rangle- \hat{q}\hat\varphi(v)\\
& \geq c_6\Vert\nabla v\Vert_p^p +\left( 1- \frac{\hat{q}}{p}\right)\int_\Omega\alpha |v|^p dx+\int_\Omega\left(\hat{q} F(x,v)-f(x,v)v \right) dx\\
& \geq c_6\Vert\nabla v\Vert_p^p+k\Vert v\Vert_p^p-C_3\Vert v\Vert_r^r>0
\end{split}
\end{equation*}
provided $0<\Vert u\Vert\leq\rho$, with $\rho$ small enough. Thus,
\begin{equation}\label{posder}
\frac{d}{dt} \hat{\varphi}(t u) \lfloor_{t = t_0} > 0\;\;\text{whenever}\;\; t_0\in\ ]0,1],\;\hat{\varphi}(t_0 u) = 0.
\end{equation}
Since $\hat{\varphi}(t_2 u) = 0$, from \eqref{posder} it follows 
\begin{equation}\label{poshatphi}
\hat{\varphi}(tu)>0\quad\forall\, t \in\ ]t_2,t_2+\delta_1],
\end{equation}
where $0<\delta_1<t_1 - t_2$. Letting
$$\hat{t}:=\left\{
\begin{array}{ll}
\min\{t\in[t_2 + \delta_1,t_1]:\hat{\varphi}(tu) = 0\} & \text{ if }\{ t\in[t_2 + \delta_1,t_1]:\hat{\varphi}(tu) = 0\}\neq\emptyset,\\
1 & \text{ otherwise,}
\end{array}
\right.$$
\eqref{poshatphi} forces $\hat{t}>t_2+\delta_1 $. So, via \eqref{posder} when $\hat\varphi(\hat tu)=0$, we can find a $\delta_2 \in\ ]0,\hat{t}-t_2-\delta_1[$ such that
\begin{equation}\label{neghatphi}
\hat{\varphi}(t u) < 0,\quad t\in [\hat{t}-\delta_2,\hat{t}[\ .
\end{equation}
Now, by \eqref{poshatphi}, \eqref{neghatphi}, and Bolzano's theorem, $\hat{\varphi}(t^*u) = 0$ for some $t^*\in\ ]t_2 +\delta_1, \hat t-\delta_2[$, which is impossible due to the choice of $\hat{t}$. Therefore, $t_1\leq t_2$, as desired.\\
One actually has $t_1=t_2$, because assuming $t_1 < t_2$ leads to $\hat{\varphi}(tu) = 0$ in $]t_1,t_2[$, against \eqref{posder}. Put $t(u):=t_1=t_2$. Evidently,
$$\hat{\varphi}(tu) < 0\quad\forall\, t\in\ ]0,t(u)[\ ,\quad\hat\varphi(t(u)u)=0,
\quad\hat{\varphi}(tu) > 0\quad\forall\, t\in\ ]t(u),1]  ,$$
whence the map $r(u):= t(u)u$, $u\in\overline{B}_\rho\setminus\{0\}$, turns out continuous,
$$r(\overline{B}_\rho\setminus\{0\})\subseteq(\hat\varphi^0\cap\overline{B}_\rho)\setminus \{0\},\quad 
r\lfloor_{(\hat\varphi^0\cap\overline{B}_\rho)\setminus \{0\}}=id\lfloor_{(\hat\varphi^0\cap\overline{B}_\rho)\setminus \{0\}}.$$
This shows that $(\hat\varphi^0\cap\overline{B}_\rho)\setminus \{0\}$ is a retract of $\overline{B}_\rho\setminus\{0\}$. Consequently, $(\hat\varphi^0\cap\overline{B}_\rho)\setminus\{0\}$ turns out contractible in itself, since $\overline{B}_\rho\setminus\{0\}$ enjoys the same property. Now, Propositions 4.9--4.10 of \cite{GD} give
$$H_k(\hat\varphi^0\cap\overline{B}_\rho,(\hat\varphi^0\cap\overline{B}_\rho)\setminus\{0\})=0,\quad k\in\mathbb{N}_0\, ,$$
i.e., the conclusion.
\end{proof}
We are now ready to establish our first existence result.
\begin{theorem}\label{thmone}
If $({\rm a}_1)$--$({\rm a}_4)$, \eqref{abeta}, and $({\rm f}_1)$--$({\rm f}_4)$ hold true then \eqref{prob} admits a nodal solution $\hat u\in C^1(\overline{\Omega})\cap[v_-,u_+]$.
\end{theorem}
\begin{proof}
Since $\hat\varphi$ turns out coercive, via \cite[Proposition 6.64]{MoMoPa} one has $C_k(\hat\varphi,\infty)=\delta_{k,0}{\mathbb Z}$. Combining $({\rm j}_2)$ of Lemma \ref{prophatphi} with \cite[Example 6.45]{MoMoPa} next entails 
\begin{equation}\label{prophatphi0}
C_k(\hat\varphi,u_+)=C_k(\hat\varphi,v_-)=\delta_{k,0}{\mathbb Z}.
\end{equation}
Suppose $K(\hat\varphi)=\{0,u_+,v_-\}$, recall Lemma \ref{prophatphix}, and write the Morse relation \eqref{morse} for $t:=-1$, to arrive at $2(-1)^0=(-1)^0$, which is evidently impossible. Thus, there exists a point $\hat u\in K(\hat\varphi)\setminus\{0,u_+,v_-\}$. The conclusion easily stems from $({\rm j}_1)$ of Lemma \ref{prophatphi} besides \eqref{defhatf}--\eqref{defhatbeta}.
\end{proof}
\subsection{The $(q-1)$-linear case}
For $q, c_7$ given by $({\rm a}_4)$, $\alpha_0:=|\alpha|/c_7$, and $\beta_0:=\beta/c_7$, assume that:
\begin{itemize}
\item[$({\rm f}_5)$] Uniformly in $x\in\Omega$, one has
$$c_7\hat\lambda_2(q,\alpha_0,\beta_0)<c_{10}<\liminf_{t\to 0}\frac{f(x,t)}{|t|^{q-2}t}\leq\limsup_{t\to 0}\frac{f(x,t)}{|t|^{q-2}t}\leq c_{11}.$$
\end{itemize}
A careful inspection of proofs reveals that all the auxiliary results above, except Lemma \ref{prophatphix}, remain valid whenever $({\rm f}_5)$ replaces $({\rm f}_3)$. So, although critical groups cannot be employed, the same conclusion is achieved via $({\rm p}_3)$ in Section \ref{sectiontwo}.
\begin{theorem}\label{thmtwo}
Let  $({\rm a}_1)$--$({\rm a}_4)$, \eqref{abeta}, $({\rm f}_1)$--$({\rm f}_2)$, and $({\rm f}_5)$ be satisfied. Then \eqref{prob}
possesses a nodal solution $\hat u\in C^1(\overline{\Omega})$.
\end{theorem}
\begin{proof}
Recalling $({\rm j}_1)$--$({\rm j}_2)$ of Lemma \ref{prophatphi}, we may suppose $K(\hat\varphi)$ finite, the local minimizer $v_-,u_+$ proper, besides $\hat\varphi(v_-)\leq\hat \varphi(u_+)$ (the other case is analogous). If $0<\rho<\Vert u_+-v_-\Vert$ complies with
\begin{equation}\label{crho}
\hat\varphi(u_+)<C_\rho:=\inf_{u\in\partial B_\rho(u_+)}\hat\varphi(u)
\end{equation}
then the Mountain Pass theorem produces a point $u_1\in K(\hat\varphi)$ such that
\begin{equation}\label{defw}
C_\rho\leq\hat\varphi(u_1)= \inf_{\gamma\in\Gamma}\max_{t\in[0,1]}\hat\varphi(\gamma(t)),
\end{equation}
where
$$\Gamma:=\{\gamma\in C^0([0,1],X):\;\gamma(0)=v_-,\;\gamma(1)=u_+\}\, .$$
By $({\rm j}_1)$ in Lemma \ref{prophatphi}, $u_1$ belongs to $C^1(\overline{\Omega})$ and solves \eqref{prob}. Through \eqref{crho}--\eqref{defw} we next get $u_1\neq v_-,u_+$. Thus, on account of Proposition \ref{propone}, it remains to check whether $u_1\neq 0$. This will follow from the inequality $\hat\varphi(u_1)<0$, which evidently holds once there exists a path
$\tilde\gamma \in\Gamma$ such that
\begin{equation}\label{path}
\hat\varphi(\tilde\gamma(t))<0\quad\forall\, t\in [0,1].
\end{equation}
Fix $\ep>0$. Using $({\rm a}_4)$ yields
\begin{equation}\label{Gsmall} 
G(\xi)\leq\frac{c_7+\ep}{q}|\xi|^q,\quad |\xi|\leq\delta,
\end{equation}
while $({\rm f}_5)$ entails
\begin{equation}\label{Fsmall}
F(x,t)\geq\frac{c_{10}}{q}|t|^q,\quad (x,t)\in\Omega\times[-\delta,\delta],
\end{equation}
provided $\delta>0$ is small enough. Now, denote by ${\cal E}^0_q$ the functional \eqref{defE} with $\alpha_0$ and $\beta_0$ in place of $\alpha$ and $\beta$, respectively. Thanks to $({\rm p}_3)$ besides Lemma \ref{density}, given $\eta>0$ we can find a
path $\hat\gamma_\eta\in\hat\Gamma_C$ such that
\begin{equation}\label{smalleta}
\max_{t\in [-1,1]}c_7{\cal E}_q^0(\hat\gamma_\eta(t))<c_7\hat\lambda_2(q,\alpha_0,\beta_0)+\eta.
\end{equation}
Since $\hat\gamma_\eta([-1,1])$ is compact in $C^1(\overline{\Omega})$ while $-v_-,u_+ \in{\rm int}(C_+)$, there exists
$\tau>0$ fulfilling
$$v_-(x)\leq\tau\hat\gamma_\eta(t)(x)\leq u_+(x),\quad |\tau\hat\gamma_\eta(t)(x)|\leq\delta\leq 1,\quad |\tau\nabla\hat\gamma_\eta(t)(x)|\leq\delta$$
for all $(x,t)\in\Omega\times [-1,1]$. On account of $q\leq p$,  the inequalities above, \eqref{Gsmall}--\eqref{smalleta}, and $\Vert\hat\gamma_\eta(t)\Vert_q\equiv 1$, one arrives at
\begin{equation*}
\hat\varphi(\tau\hat\gamma_\eta(t))\leq\frac{\tau^q}{q}\left[ c_7{\cal E}_q^0(\hat\gamma_\eta(t))+\ep\Vert\nabla\hat\gamma_\eta(t)\Vert_q^q-c_{10}\right]<\frac{\tau^q}{q}\left[ c_7\hat\lambda_2(q,\alpha_0,\beta_0)+\eta+C_1\,\ep-c_{10}\right].
\end{equation*}
Therefore,
\begin{equation}\label{middle}
\hat\varphi(\tau\hat\gamma_\eta(t))<0\quad\forall\, t\in [-1,1]
\end{equation}
as soon as $\ep$ and $\eta$ are taken so small that
$$\eta+C_1\,\ep<c_{10}- c_7\hat\lambda_2(q,\alpha_0,\beta_0);$$
see $({\rm f}_5)$. Next, if $\hat a:=\hat\varphi_+(u_+)$ then $\hat a<0$, because
$$\hat\varphi_+(u_+)=\inf_{u\in X}\hat\varphi_+(u)<0=\hat\varphi_+(0),$$
no critical value of $\hat\varphi_+$ lies in $(\hat a,0)$, and $K_{\hat a}(\hat\varphi_+) =\{ u_+\}$; see Lemma \ref{prophatphi}.
Thus, the second deformation lemma \cite[Theorem 5.1.33]{GaPa-NA2006} gives a continuous map $h:[0,1]\times (\hat\varphi_+^0\setminus\{0\})\to\hat\varphi_+^0$ satisfying
\begin{equation}\label{proph}
h(0,u)=u\, ,\quad h(1,u)=u_+\, ,\quad\hat\varphi_+(h(t,u))\leq\hat\varphi_+(u)
\end{equation}
for all $(t,u)\in [0,1]\times(\hat\varphi_+^0\setminus\{0\})$. By \eqref{middle} one has
$$\hat\varphi_+(\tau\hat u_1(q,\alpha_0,\beta_0))=\hat\varphi(\tau\hat\gamma_\eta(1))<0.$$
Hence, it makes sense to define
$$\gamma_+(t):=h(t,\tau\hat u_1(q,\alpha_0,\beta_0))^+,\quad t\in [0,1].$$
The path $\gamma_+:[0,1]\to X$ connects $\tau\hat u_1(q,\alpha_0,\beta_0)$ with $u_+$. Moreover, due to \eqref{middle}--\eqref{proph},
\begin{equation}\label{gammapiu}
\hat\varphi(\gamma_+(t))=\hat\varphi_+(\gamma_+(t))\leq\hat\varphi_+(\tau\hat u_1(q,\alpha_0,\beta_0))=\hat\varphi(\tau\hat\gamma_\eta(1))<0\quad\forall\, t\in [0,1].
\end{equation}
A similar reasoning, where $\hat\varphi_-$ takes the place of $\hat\varphi_+$, produces a continuous function $\gamma_-:[0,1]\to X$ such that $\gamma_-(0)=v_-$, $\gamma_-(1)=-\tau\hat u_1(q,\alpha_0,\beta_0)$, as well as
\begin{equation}\label{gammameno}
\hat\varphi(\gamma_-(t))<0\;\;\text{in}\;\;[0,1].
\end{equation}
Concatenating $\gamma_-$, $\tau\hat\gamma_\eta$, and $\gamma_+$ one obtains a path $\tilde\gamma \in\Gamma$ which, in view of \eqref{gammapiu}--\eqref{gammameno}, besides \eqref{middle}, fulfills (\ref{path}).
\end{proof}
\subsection{The case of $(p,2)$-Laplacian}
Let $p>2$, $a_0(t):=t^{p-2}+1$, namely $q:=2$, $f(x,\cdot)\in C^1([-\theta,\theta])$ for every $x\in\Omega$, and let
$\hat\lambda_n$, $E(\hat\lambda_n)$, $\bar H_n$, $\hat H_n$ be like at the end of Section \ref{sectiontwo}. The following assumptions will be posited.
\begin{itemize}
\item[ $({\rm f}_6)$] There exists $a_\theta\in L^\infty(\Omega)$ such that $|f'_t(x, t)|\leq a_\theta(x)$ in $\Omega\times[-\theta,\theta]$.
\item[ $({\rm f}_7)$] With appropriate $m\geq 2$, $\delta_0>0$ small, $b\in L^\infty(\Omega)\setminus\{\hat\lambda_{m+1}\}$ one has $f(x,t)t\geq\hat\lambda_mt^2$ for all $(x,t)\in\Omega\times [-\delta_0,\delta_0]$ and 
$$f'_t(x,0)=\lim_{t\to 0}\frac{f(x,t)}{t}\leq b(x)\leq\hat\lambda_{m+1}\;\;\text{uniformly in}\;\; x\in\Omega.$$
\end{itemize}
\begin{remark}\label{furtherhyp}
As before, except Lemma \ref{prophatphix}, the auxiliary results above remain valid if $({\rm f}_6)$--$({\rm f}_7)$ replace $({\rm f}_3)$.
\end{remark}
Now, recall \eqref{defhatf}--\eqref{defhatbeta} and define
\begin{equation*}
\hat\psi(u):=\frac{1}{2}\left[ \Vert\nabla u\Vert_2^2+\int_\Omega(\alpha+\hat\alpha)u^2dx\right]+\int_{\partial\Omega}\hat B(x,u)d\sigma-\int_\Omega\hat F(x,u)dx,\;\; u\in H^1(\Omega).
\end{equation*}
Evidently, $\hat\psi$ is $C^2$ in a neighborhood of the origin, besides $C^1$ on the whole $H^1(\Omega)$.
\begin{lemma}\label{prophatpsi}
Hypotheses \eqref{abeta} and $({\rm f}_6)$-- $({\rm f}_7)$ entail $C_k(\hat\psi,0)=\delta_{k,d_m}\mathbb{Z}$ for every $k\in\mathbb{N}_0$, where $d_m:={\rm dim}(\bar H_m)\geq 2$.
\end{lemma}
\begin{proof}
Since $\bar H_m$ is finite dimensional, we can find $\rho_1>0$ such that
$$u\in\bar H_m\cap\overline{B}_{\rho_1}\;\;\implies\;\; |u(x)|\leq\delta_0\;\;\forall\, x\in\Omega.$$
Via $({\rm f}_7)$ and $({\rm p}_4)$, this easily leads to
\begin{equation}\label{loclinkone}
\hat\psi(u)\leq\frac{1}{2}\left[{\cal E}_2(u)-\hat\lambda_m\Vert u\Vert_2^2\right]\leq 0,\;\; u\in\bar H_m\cap
\overline{B}_{\rho_1}.
\end{equation}
Next, given $\ep>0$, $r>2$, assumption $({\rm f}_7)$ yields
\begin{equation*}
\hat F(x,t)\leq\frac{1}{2}\left(b(x)+\ep\right)t^2+C_1|t|^r\;\;\text{in}\;\;\Omega\times\R,
\end{equation*}
whence, by \cite[Lemma 2.2]{DaMaPa-JMAA2016},
\begin{equation*}
\hat\psi(u)\geq\frac{1}{2}\left[{\cal E}_2(u)-\int_\Omega b(x)u^2 dx-\ep\Vert u\Vert^2\right]-C_1\Vert u\Vert^r\geq\frac{\hat c-\ep}{2}\Vert u\Vert^2-C_1\Vert u\Vert^r\;\;\forall\, u\in\hat H_{m+1}.
\end{equation*}
Here, $\Vert\cdot\Vert$ denotes the usual norm of $H^1(\Omega)$. Choosing $\ep<\hat c$ we thus achieve
\begin{equation}\label{loclinktwo}
\hat\psi(u)>0,\;\; u\in\hat H_{m+1}\cap\overline{B}_{\rho_2}\setminus\{0\},
\end{equation}
provided $\rho_2>0$ is small enough. Inequalities \eqref{loclinkone}--\eqref{loclinktwo} ensure that $\hat\psi$ has a local linking at zero with respect to the sum decomposition $H^1(\Omega)=\bar H_m\oplus V$, where $V$ indicates the closure of $\hat H_{m+1}$ in $H^1(\Omega)$. Since $\psi$ is coercive, it satisfies condition (PS); see Proposition \ref{coercivePS}. So, the conclusion follows from \cite[Proposition 2.3]{Su-NA2002}.
\end{proof}
\begin{theorem}\label{ptwocase}
Let \eqref{abeta}, $({\rm f}_1)$, $({\rm f}_6)$, and $({\rm f}_7)$ be satisfied. Then \eqref{prob}, where $p>2$ while $a_0(t):=t^{p-2}+1$, admits a nodal solution $\hat u\in C^1(\overline{\Omega})$.
\end{theorem}
\begin{proof}
Set $\psi:=\hat\psi\lfloor_{X}$. One evidently has $C_k(\psi,0)=C_k(\hat\psi,0)$, because $X \hookrightarrow H^1(\Omega)$ densely. Consequently, thanks to Lemma \ref{prophatpsi},
\begin{equation}\label{proppsi}
C_k(\psi,0)=\delta_{k,d_m}\mathbb{Z}.
\end{equation}
Observe next that
\begin{equation*}
|\hat\varphi(u)-\psi(u)|\leq C_1\Vert u\Vert^p,\quad |\langle\hat\varphi'(u)-\psi'(u),v\rangle|\leq C_2\Vert u\Vert^{p-1}\Vert v\Vert\quad\forall\, u,v\in X,
\end{equation*}
as a simple computation shows. Therefore, the $C^1$- continuity of critical groups \cite[Theorem 5.126]{GaPa-Ex2} and \eqref{proppsi} produce
\begin{equation}\label{grouphatvarphi1}
C_k(\hat\varphi,0)=\delta_{k,d_m}\mathbb{Z},\;\; k\in\mathbb{N}_0.
\end{equation}
On the other hand, $\hat\varphi$ is coercive, whence $\inf_{u\in X}\hat\varphi(u)>-\infty$, and fulfills (PS). By \cite[Proposition 6.64]{MoMoPa} we thus get
\begin{equation}\label{grouphatvarphi2}
C_k(\hat\varphi,\infty)=\delta_{k,0}\mathbb{Z},\;\; k\in\mathbb{N}_0.
\end{equation}
Combining \eqref{grouphatvarphi1}--\eqref{grouphatvarphi2} with \cite[Corollary 6.92]{MoMoPa} one arrives at
\begin{equation}\label{grouphatvarphi3}
C_{d_m-1}(\hat\varphi,\hat u)\neq 0\;\;\text{or}\;\; C_{d_m+1}(\hat\varphi,\hat u)\neq 0,\;\;\text{where}\;\; d_m\geq 2,
\end{equation}
for some $\hat u\in K(\hat\varphi)\setminus\{0\}$. Now, the conclusion easily stems from \eqref{prophatphi0}, \eqref{grouphatvarphi3}, besides $({\rm j}_1)$ in Lemma \ref{prophatphi}; see also Remark \ref{nodsol}.
\end{proof}
\section{Nodal solutions: multiplicity}\label{sectionfour}
Under a symmetry condition on $f(x,\cdot)$, problem \eqref{prob} possesses infinitely many sign-changing solutions.
\begin{theorem}\label{sequencesol}
Suppose $({\rm a}_1)$--$({\rm a}_4)$, \eqref{abeta}, and $({\rm f}_1)$--$({\rm f}_3)$ hold. If, moreover,
\begin{itemize}
\item[$({\rm f}_8)$] the function $t\mapsto f(x,t)$ is odd in $[-\theta,\theta]$  for every $x\in\Omega$
\end{itemize}
then there exists a sequence $\{u_n\}\subseteq C^1(\overline{\Omega})$ of distinct nodal solutions to \eqref{prob} satisfying $u_n\to 0$ in $C^1(\overline{\Omega})$.
\end{theorem}
\begin{proof}
The proof is patterned after that of \cite[Theorem 4.3]{MaPa-RLM2018}; so, we only sketch it. Via $({\rm a}_4)$ one has
\begin{equation}\label{m1}
G(\xi)\leq\frac{C_1}{q}|\xi|^q,\;\; |\xi|\leq\delta,
\end{equation}
while, given $\eta>0$, assumption $({\rm f}_3)$ entails
\begin{equation}\label{m2}
F(x,t)\geq\frac{\eta}{q}|t|^q,\;\; (x,t)\in\Omega\times[-\delta,\delta],
\end{equation}
with $\delta>0$ small enough. Let $V\subseteq X$ be any finite dimensional subspace and let $\rho>0$ fulfill 
\begin{equation}\label{m3}
u\in V\cap\overline{B}_\rho\;\;\implies\;\; |u(x)|\leq\delta\;\;\forall\, x\in\Omega.
\end{equation}
Gathering \eqref{m1}--\eqref{m3} together leads to
\begin{equation*}
\hat\varphi(u)\leq\frac{C_1}{q}\Vert\nabla u\Vert_q^q+\frac{1}{p}\left(\int_\Omega|\alpha|\,|u|^pd x+\int_{\partial\Omega}\beta |u|^pd\sigma\right)-\frac{\eta}{q}\Vert u\Vert^q_q\leq (C_2-C_3\eta)\Vert u\Vert^q_q<0
\end{equation*}
provided $u\in(V\cap\overline{B}_\rho)\setminus\{0\}$, $\eta>C_2/C_3$. Here, the equivalence between all norms on $V$ was also exploited. Hence, Theorem 1 of \cite{Ka-JFA2005} furnishes a sequence 
\begin{equation}\label{incl1}
\{u_n\}\subseteq K(\hat\varphi)\cap\hat\varphi^{-1}(]-\infty,0[)
\end{equation}
that converges to zero in $X$. Through standard arguments from the nonlinear regularity theory we actually have $\{u_n\}\subseteq C^1( \overline{\Omega})$ as well as $\Vert u_n\Vert_{C^1( \overline{\Omega})}\to 0$. Now, assertion $({\rm j}_1)$ of Lemma \ref{prophatphi}, besides \eqref{incl1}, easily yield the conclusion.
\end{proof}
When the left-hand side is the $(p,2)$-Laplacian, one can do without symmetry. However, a further condition on  $f$ will be imposed.
\begin{itemize}
\item[$({\rm f}_9)$] There exists $\mu_\theta>0$ such that $t\mapsto f(x,t)+\mu_\theta|t|^{p-2}t$ is non-decreasing on $[-\theta,\theta]$ for all $x\in\Omega$.
\end{itemize}
\begin{theorem}\label{spcaseone}
Let \eqref{abeta}, $({\rm f}_1)$--$({\rm f}_3)$, and $({\rm f}_9)$ be satisfied. If $p>2$, $a_0(t):=t^{p-2}+1$, while $\beta>0$ on $\partial\Omega$ then \eqref{prob} possesses two nodal solutions $\hat u,\tilde u\in C^1(\overline{\Omega})$. 
\end{theorem}
\begin{proof}
A first sign-changing function $\hat u\in C^1(\overline{\Omega})\cap[v_-,u_+]$ that solves \eqref{prob} directly comes from Theorem \ref{thmone}. Since $p>2$ and $a(\xi)=(|\xi|^{p-2}+1)\xi$, an easy computation shows that
$$(\nabla a(\xi)y)\cdot y\geq |y|^2\quad\forall\,\xi,y\in\R^N.$$
Hence, the tangency principle \cite[Theorem 2.5.2]{PS} gives
\begin{equation}\label{doubleineq}
v_-<\hat u<u_+\;\;\text{in}\;\;\Omega.
\end{equation}
Now, pick any $\mu>\mu_\theta$ and define
$$h_1:=f(\cdot,\hat u)+\mu_\theta|\hat u|^{p-2}\hat u+(\mu-\mu_\theta)|\hat u|^{p-2}\hat u,\;\;
h_2:=f(\cdot,u_+)+\mu_\theta u_+^{p-1}+(\mu-\mu_\theta)u_+^{p-1}.$$
Via $({\rm f}_1)$ we obtain $h_1,h_2\in L^\infty(\Omega)$ while $({\rm f}_9)$ entails $h_1\leq h_2$. Thanks to \eqref{doubleineq}, for every compact set $K\subseteq\Omega$ one has $\essinf_{x\in K}(h_2(x)-h_1(x))>0$. The condition on $\beta$ forces
$$\frac{\partial u_+}{\partial n_a}\big\lfloor_{\partial\Omega}=-\beta u_+^{p-1}<0.$$
Consequently, by \cite[Proposition 3]{FMP}, $u_+-\hat u\in{\rm int}(C_+)$. A quite similar reasoning produces $\hat u-v_-\in{\rm int}(C_+)$, whence, a fortiori, $\hat u\in{\rm int}_{C^1(\overline{\Omega})}([v_-,u_+])$. At this point, we adapt the flow invariance arguments made in \cite{HHLL} to get a nodal solution $\tilde u\in C^1(\overline{\Omega})\setminus 
{\rm int}_{C^1(\overline{\Omega})}([v_-,u_+])$ of \eqref{prob}. It is evident that $\tilde u\neq\hat u$.
\end{proof}
A better situation occurs in the semi-linear case $p:=2$ and $a_0(t):=1$, because the regularity theory of \cite{Wa} allows to weaken \eqref{abeta} as follows.
\begin{equation}\label{abetaw}
a\in L^s(\Omega)\;\;\text{for some}\;\; s>N,\;\; a^+\in L^\infty(\Omega),\;\; \beta\in W^{2,\infty}(\Omega),\;\;\text{and}\;\;\beta\geq 0.
\end{equation}
Write $\hat m:=\max\{n_0,2\}$, where $n_0:=\inf\{n\in\mathbb{N}:\hat\lambda_n>0\}$.
\begin{theorem}\label{spcasetwo}
If $p:=2$, $a_0(t):=1$, and \eqref{abetaw}, $({\rm f}_1)$, $({\rm f}_6)$, $({\rm f}_7)$ with $m\geq\hat m$ hold true then \eqref{prob} admits three nodal solutions $\hat u,\bar u,\tilde u\in C^1(\overline{\Omega})$.
\end{theorem}
\begin{proof}
The same technique exploited to prove both \cite[Theorem 3.2]{MaPa-JMAA2016} and \cite[Theorem 3.2]{DaMaPa-JMAA2016} yields a point $\hat u\in K(\hat\varphi)$ of mountain pass type. So, 
\begin{equation}\label{morse0}
C_k(\hat\varphi, \hat u)=\delta_{k,1}\mathbb{Z}\quad\forall\, k\in\mathbb{N}_0;
\end{equation}
vide \cite[Lemma 3.1]{MaPa-JMAA2016}. Recalling Remarks \ref{nodsol}--\ref{furtherhyp}, conclusions $({\rm j}_1)$-- $({\rm j}_2)$ in Lemma \ref{prophatphi} ensure that $\hat u\in C^1(\overline{\Omega})$ is a nodal solution to \eqref{prob}. Via $({\rm f}_6)$ we can find a
$\mu_\theta>0$ such that $t\mapsto f(x,t)+\mu_\theta t$ turns out non-decreasing on $[-\theta,\theta]$ for any $x\in\Omega$. Since $\hat u\leq u_+$, one has
\begin{eqnarray*}
-\Delta \hat u(x)+[a(x)+\mu_\theta]\hat u(x)=f(x,\hat u(x))+\mu_\theta\hat u(x)\\
\leq f(x,u_+(x))+\mu_\theta u_+(x)=-\Delta u_+(x)+[a(x)+\mu_\theta] u_+(x),
\end{eqnarray*}
which entails
\begin{equation*}
\Delta (u_+-\hat u)(x)\leq(\Vert a^+\Vert_\infty+\mu_\theta)[u_+(x)-\hat u(x)],\;\; x\in\Omega.
\end{equation*}
From the strong maximum principle \cite[p. 34]{PS} it follows $u_+-\hat u>0$ in $\Omega$. Suppose $(u_+-\hat u)(x_0)=0$ for some $x_0\in\partial\Omega$. The boundary point lemma \cite[p. 120]{PS} leads to $\frac{\partial(u_+-\hat u)}{\partial n}(x_0)<0$, whence
\begin{equation*}
-\beta(x_0)u_+(x_0)=\frac{\partial u_+}{\partial n}(x_0)<\frac{\partial\hat u}{\partial n}(x_0)=-\beta(x_0)\hat u(x_0).
\end{equation*}
However, this is impossible, because $\beta\geq 0$. Thus, $u_+-\hat u>0$ on the whole $\overline{\Omega}$, and $u_+-\hat u\in{\rm int}(C_+)$. An analogous reasoning produces $\hat u-v_-\in{\rm int}(C_+)$. Hence, 
\begin{equation}\label{prophatu}
\hat u\in {\rm int}_{C^1(\overline{\Omega})}([v_-,u_+]).
\end{equation}
A further nodal solution $\bar u\in C^1(\overline{\Omega})$ of \eqref{prob} is easily obtained. Indeed, assertion $({\rm j}_2)$ in Lemma \ref{prophatphi} forces
\begin{equation}\label{morse1}
C_k(\hat\varphi,u_+)=C_k(\hat\varphi,v_-)=\delta_{k,0}\mathbb{Z},
\end{equation}
while
\begin{equation}\label{morse2}
C_k(\hat\varphi,0)=\delta_{k,d_m}\mathbb{Z},\;\;\text{and}\;\; C_k(\hat\varphi,\infty)=\delta_{k,0}\mathbb{Z};
\end{equation}
cf. the proof of Theorem \ref{ptwocase}. Now, if $K(\hat\varphi)=\{0,u_+,v_-,\hat u\}$ then, combining \eqref{morse0}, \eqref{morse1}, and \eqref{morse2} with \eqref{morse} we would immediately reach a contradiction. So, there exists $\bar u\in
K(\hat\varphi)\setminus\{0,u_+,v_-,\hat u\}$. As already shown for $\hat u$, one has $\bar u\in C^1\overline{\Omega})$, $\bar u$ nodal, besides
\begin{equation}\label{proptildeu}
\bar u\in {\rm int}_{C^1(\overline{\Omega})}([v_-,u_+]).
\end{equation}
Finally, adapting the flow invariance arguments made in \cite{HHLL} we get a nodal solution
\begin{equation}\label{propbaru}
\tilde u\in C^1(\overline{\Omega})\setminus{\rm int}_{C^1(\overline{\Omega})}([v_-,u_+])
\end{equation}
to problem \eqref{prob}. From \eqref{prophatu}, \eqref{proptildeu}, and \eqref{propbaru} it evidently follows $\tilde u\not\in\{\hat u,\bar u\}$. 
\end{proof}
\begin{remark}
Let us note that $\tilde u\not\in K(\hat\varphi)$, otherwise, thanks to Lemma \ref{prophatphi} and the trick at the beginning of the above proof, $\tilde u\in {\rm int}_{C^1(\overline{\Omega})}([v_-,u_+])$, against \eqref{propbaru}.
\end{remark}
\section*{Acknowledgement}
This work is performed within the 2016--2018 Research Plan - Intervention Line 2: `Variational Methods and Differential Equations', and partially supported by GNAMPA of INDAM.
\end{document}